\documentclass[12pt,twoside]{amsart}
\usepackage{amsmath, amssymb, paralist,tabularx,supertabular,amsmath, verbatim, doc, amsthm, amssymb, mathrsfs, manfnt}
\usepackage[all]{xy}
\evensidemargin \oddsidemargin
\parskip=\medskipamount
\DeclareMathOperator*{\f}{\textbf{f}}
\DeclareMathOperator*{\g}{\textbf{g}}
\DeclareMathOperator*{\p}{\mathfrak{p}}
\DeclareMathOperator*{\cp}{\mathcal{P}}
\DeclareMathOperator*{\cn}{\mathcal{N}}

\DeclareMathOperator*{\ch}{\mathcal{H}}

\DeclareMathOperator*{\cl}{\mathcal{L}}

\DeclareMathOperator*{\diff}{\mathfrak{d}}
\DeclareMathOperator*{\diffinv}{\mathfrak{d}^{-1}}
\DeclareMathOperator*{\cond}{\mathfrak{f}}

\DeclareMathOperator*{\q}{\mathfrak{q}}
\DeclareMathOperator*{\qz}{\mathfrak{q}_0}
\DeclareMathOperator*{\td}{\tilde{\mathfrak{d}}}
\DeclareMathOperator*{\tf}{\tilde{\mathfrak{f}}}
\DeclareMathOperator*{\tl}{\tilde{\mathcal{L}}}
\DeclareMathOperator*{\qo}{\mathfrak{q}_1}

\DeclareMathOperator*{\dete}{det}
\DeclareMathOperator*{\sltwo}{SL_2}

\newtheorem{thm}{Theorem}[section]
\newtheorem{theorem}[thm]{Theorem}

\newtheorem{lemma}[thm]{Lemma}	
\newtheorem{proposition}[thm]{Proposition}

\theoremstyle{definition}
\newtheorem{definition}[thm]{Definition}

\theoremstyle{remark}

\newtheorem*{lemma*}{Lemma}

\numberwithin{equation}{section}

\title{Decomposition Theorems for Hilbert Modular Newforms}

\author{Benjamin Linowitz}
\address{Department of Mathematics\\ 
6188 Kemeny Hall\\
Dartmouth College\\
Hanover, NH 03755, USA}
\email[] {benjamin.D.linowitz@dartmouth.edu}

\begin{document}

\begin{abstract} 
Let $\mathscr{S}_k^+(\cn,\Phi)$ denote the space generated by Hilbert modular newforms (over a fixed totally real field $K$) of weight $k$, level $\cn$ and Hecke character $\Phi$. We show how to decompose $\mathscr{S}_k^+(\cn,\Phi)$ into direct sums of twists of other spaces of newforms. This sheds light on the behavior of a newform under a character twist: the exact level of the twist of a newform, when such a twist is itself a newform, and when a newform may be realized as the twist of a primitive newform. These results were proven for elliptic modular forms by Hijikata, Pizer and Shemanske by employing a formula for the trace of the Hecke operator $T_k(n)$. We obtain our results not by employing a more general formula for the trace of Hecke operators on spaces of Hilbert modular forms, but instead by using basic properties of newforms which were proven for elliptic modular forms by Li, and Atkin and Li, and later extended to Hilbert modular forms by Shemanske and Walling.
\end{abstract}

\maketitle

%%%%%%%%%%%%%%%%%%%%%
%%%%%%%%%%%%%%%%%%%%%
%%%%%%%%%%%%%%%%%%%%%
%%%%%%%%%%%%%%%%%%%%%
%%%%%%%%%%%%%%%%%%%%%
%%%%%%%%%%%%%%%%%%%%%
%%%%%%%%%%%%%%%%%%%%%
%%%%%%%%%%%%%%%%%%%%%

\section{Introduction}

In their paper \textit{Twists of Newforms} Hijikata, Pizer and Shemanske \cite{HPS} show how to decompose spaces of elliptic modular newforms into direct sums of character twists of other spaces of newforms. These decompositions provide important information about the behavior of newforms under character twists; for example, the exact level of the twist of a newform, when such a twist is itself a newform, and when a newform may be realized as the twist of a primitive newform. The main technique they used to prove these decompositions is Hijikata's formula  \cite{hijikata} for the trace of the Hecke operator $T_k(n)$ acting on the space of cusp forms $S_k(N, \phi)$ of weight $k$, level $N$ and character $\phi$.

Fix a positive integer $N$. The Hecke algebra spanned by the $T_k(n)$ with $n$ coprime to $N$ acting on $S_k(N,\phi)$ is semi-simple. Showing that two Hecke-modules $A$ and $B$ are isomorphic therefore reduces to showing that the trace of $T_k(n)$ on $A$ equals the trace of $T_k(n)$ on $B$ for all $n$ coprime to $N$. It is in this context that Hijikata's formula is applied. For instance, in Theorem 3.2 they take $A$ to be the space $S^0_k(N,\omega\phi)$ generated by newforms of level $N$ and character $\omega\phi$ and $B$ to be the space $S_k^0(N,\overline{\omega}\phi)^{\omega}$ generated by twists (by $\omega$) of newforms of level $N$ and character $\overline{\omega}\phi$. Here $\omega$ is a Dirichlet character modulo a power of a prime dividing $N$ and $\phi$ is a Dirichlet character whose conductor is coprime to the conductor of $\omega$. Hijikata, Pizer and Shemanske use Hijikata's formula for the trace of $T_k(n)$ to show that $$S^0_k(N,\omega\phi)\cong S_k^0(N,\overline{\omega}\phi)^{\omega}.$$

Hijikata's formulas for the trace of Hecke operators apply in much more general contexts than modular forms on subgroups of $SL_2(\mathbb Z)$. For instance, they apply equally well to spaces of Hilbert modular forms. In theory one could use these more general formulas in order to extend the results of \cite{HPS} to the Hilbert modular setting. However the general formulas are quite complicated, so it is of interest to find a more elementary method of extending the aforementioned results. In this paper we prove several of the results of \cite{HPS} for Hilbert modular forms without appealing to formulas for the traces of Hecke operators. In fact, we use only the basic properties of newforms which were proven for elliptic modular forms in the fundamental papers \cite{li} and \cite{atkin-li} of Li, and Atkin and Li, and later extended to Hilbert modular forms by Shemanske and Walling \cite{shemanske-walling}. Thus the results of this paper are new for Hilbert modular forms over totally real number fields other than $\mathbb Q$, and provide simplified proofs for modular forms over $\mathbb Q$ (the elliptic modular case).

A sample result is the following (see Section \ref{section:prelims} for notation and terminology):

\begin{theorem}
Let $\cn$ be an integral ideal which we decompose as $\cn=\cp \cn_0$ for $\cp$ a power of a prime ideal $\p$ coprime to $\cn_0$. Set $\nu=ord_{\p}\cp$. Let $\phi$ be a numerical character modulo $\cn$ and $\Phi$ be a Hecke character extending $\phi\phi_{\infty}$ which satisfies $\frac{\nu}{2}<e(\Phi_{\cp})=ord_{\p}(\mathfrak{f}_{\Phi_{\cp}})<\nu$. Then $$\mathscr{S}_k^+(\cn,\Phi)=\bigoplus_{e(\Psi)=\nu-e(\Phi_{\cp})} \mathscr{S}_k^+(\mathfrak{p}^{e(\Phi_{\cp})}{\mathcal{N}_0},\Psi^2\Phi )^{\overline{\Psi}},$$ where the sum $\bigoplus_{e(\Psi)=\nu-e(\Phi_{\cp})} $ is taken over all $\p$-primary Hecke characters $\Psi$ with conductor $\p^{\nu-e(\Phi_{\cp})}$ and infinite part $\Psi_{\infty}(a)=\mbox{sgn}(a)^l$ for $l\in\mathbb Z^n$ and $a\in K_{\infty}^\times$.

\end{theorem}

%%%%%%%%%%%%%%%%%%%%%
%%%%%%%%%%%%%%%%%%%%%
%%%%%%%%%%%%%%%%%%%%%
%%%%%%%%%%%%%%%%%%%%%
%%%%%%%%%%%%%%%%%%%%%
%%%%%%%%%%%%%%%%%%%%%
%%%%%%%%%%%%%%%%%%%%%
%%%%%%%%%%%%%%%%%%%%%
\section{Notation and Preliminaries}\label{section:prelims}
For the most part we follow the notation of \cite{shemanske-walling, shimura-ann,shimura-duke}. However, to make this paper somewhat self-contained, we shall briefly review the basic definitions of the functions and operators which we shall study.

Let $K$ be a totally real number field of degree $n$ over $\mathbb{Q}$ with ring of integers $\mathcal{O}$, group of units $\mathcal{O}^\times$ and totally positive units $\mathcal{O}^\times_+$. Fix an embedding $a\mapsto (a^{(1)},\cdots,a^{(n)})$ of $K$ into $\mathbb{R}^n$. Let $\diff$ be the different of $K$. If $\q$ is a finite prime of $K$, we denote by $K_{\q}$ the completion of $K$ at $\q$, $\mathcal{O}_{\q}$ the valuation ring of $K_{\q}$, and $\pi_{\q}$ a local uniformizer. 

We denote by $K_A$ the ring of $K$-adeles and by $K_A^\times$ the group of $K$-ideles. As usual we view $K$ as a subgroup of $K_A$ via the diagonal embedding. If $\tilde{\alpha}\in K_A^\times$, we let $\tilde{\alpha}_{\infty}$ denote the archimedean part of $\tilde{\alpha}$ and $\tilde{\alpha}_0$ the finite part of $\tilde{\alpha}$. If $\mathcal{J}$ is an integral ideal we let $\tilde{\alpha}_{\mathcal{J}}$ denote the $\mathcal{J}$-part of $\tilde{\alpha}$. 

For an integral ideal $\cn$ we define a numerical character $\phi$ modulo $\cn$ to be a character $\phi: (\mathcal{O}/\cn)^\times \rightarrow \mathbb{C}^\times$, and a Hecke character to be a continuous character on the idele class group: $\Phi: K_A^\times/ K^\times \rightarrow \mathbb{C}^\times$. We denote the induced character on $K_A^\times$ by $\Phi$ as well. Every Hecke character is of the form $\Phi(\tilde{\alpha})=\prod_{\nu} \Phi_{\nu}(\alpha_{\nu})$ where each $\Phi_\nu$ is a character $\Phi_{\nu}: K_{\nu}^\times \longrightarrow \mathbb{C}^\times$. The conductor, $\mbox{cond}(\Phi)$, of $\Phi$ is defined to be the modulus whose finite part is $\mathfrak{f}_{\Phi}$ (see \cite{heilbronn}) and whose infinite part is the formal product of those archimedean primes $\nu$ for which $\Phi_{\nu}$ is nontrivial. In the case that $\mathfrak{f}_{\Phi}$ is a power of a single prime $\q$, we define the exponential conductor $e(\Phi)$ to be the integer such that $\mathfrak{f}_{\Phi}=\q^{e(\Phi)}$. We adopt the convention that $\phi$ and $\psi$ will always denote numerical characters and $\Phi$ and $\Psi$ will denote Hecke characters.

Let $GL_2^+(K)$ denote the group of invertible matrices with totally positive determinant and $\ch$ the complex upper half-plane. Then $GL_2^+(K)$ acts on $\ch^n$ via fractional linear transformations as follows:

$$\left(
  \begin{array}{ c c }
    a& b \\
      c&d
  \end{array} \right) \mapsto \huge\left[ \tau\rightarrow \left( \cdots, \frac{a^{(\nu)}\tau_\nu+b^{(\nu)}}{c^{(\nu)}\tau_\nu+d^{(\nu)}},\cdots \right) \huge\right]$$
  
Let $k=(k_1,...,k_n)\in \mathbb{Z}_+^n, \tau\in \ch^n$ and set 
$$(c\tau+d)^k=\prod_{\nu=1}^n (c^{(\nu)}\tau_{\nu}+d^{(\nu)})^{k_{\nu}}$$ 
and for $A\in GL_2^+(K)$
$$\dete(A)^k = \prod_{\nu=1}^n (a^{(\nu)}d^{(\nu)}-b^{(\nu)}c^{(\nu)})^{k_\nu}.$$

For $N\in \mathbb{Z}_+$, let $\Gamma_N$ denote the kernel of the reduction map $\sltwo(\mathcal{O})\rightarrow \sltwo(\mathcal{O}/N\mathcal{O}).$

Following Shimura  \cite{shimura-ann, shimura-duke}, we define $M_k(\Gamma_N)$ to be the complex vector space of functions $f$ which are holomorphic on $\ch^n$ and at the cusps of $\Gamma_N$ such that $$f(A\tau)=\dete(A)^{-\frac{k}{2}}(c\tau+d)^kf(\tau)$$
for all $A=\left(
  \begin{array}{ c c }
    a& b \\
      c&d
  \end{array} \right)\in\Gamma_N$. Let $M_k=\bigcup_{N=1}^{\infty} M_k(\Gamma_N)$.
  
  For a fractional ideal $\mathcal{I}$ and integral ideal $\cn$ we set 
  $$\Gamma_0(\cn,\mathcal{I})=\{ A\in \left(
  \begin{array}{ c c }
    \mathcal{O}& \mathcal{I}^{-1}\diffinv \\
      \cn\mathcal{I}\diff & \mathcal{O}
  \end{array} \right) : \dete A \in \mathcal{O}^\times_+ \}.$$

Let $\theta : \mathcal{O}^\times_+\rightarrow \mathbb{C}^\times$ be a character of finite order and note that there exists an element $m\in\mathbb{R}^n$ such that $\theta(a)=a^{im}$ for all totally positive $a$. While such an $m$ is not unique, we shall fix one such $m$ for the remainder of this paper. Let $\phi$ be a numerical character modulo $\cn$ and define $M_k(\Gamma_0(\cn,\mathcal{I}),\phi,\theta)$ to be the set of $f\in M_k$ which satisfy $$f(A\tau)=\dete(A)^{-\frac{k}{2}}\phi(a)\theta(\dete A) (c\tau+d)^k f(\tau)$$ for all $A=\left(
  \begin{array}{ c c }
    a& b \\
      c&d
  \end{array} \right)\in \Gamma_0(\cn,\mathcal{I})$.
  
  Fix a set of strict ideal class representatives $\mathcal{I}_1,...,\mathcal{I}_h$ of $K$, set $\Gamma_{\lambda}=\Gamma_0(\cn,\mathcal{I}_{\lambda})$, and put $$\mathfrak{M}_k(\cn,\phi,\theta)=\prod_{\lambda=1}^h M_k(\Gamma_{\lambda},\phi,\theta).$$ We are interested in studying $h$-tuples $(f_1,...,f_h)\in\mathfrak{M}_k(\cn,\phi,\theta)$.
  
  In order to deal with class number $h>1$ we follow Shimura \cite{shimura-ann, shimura-duke} and describe Hilbert modular forms as functions on an idele group. Let $G_A=GL_2(K_A)$ and view $G_K=GL_2(K)$ as a subgroup of $G_A$ via the diagonal embedding. Denote by $G_{\infty} = GL_2(\mathbb{R})^n$ the archimedean part of $G_A$ and by $G_{\infty +}$ the subgroup of elements having totally positive determinant. For an integral ideal $\cn$ of $\mathcal{O}$ and a prime $\p$, let 
$$Y_{\p}(\cn)=\{ A=\left(
  \begin{array}{ c c }
    a& b \\
      c&d
  \end{array} \right) \in  \left(
  \begin{array}{ c c }
    \mathcal{O}_{\p}& \diffinv\mathcal{O}_{\p} \\
      \cn\diff\mathcal{O}_{\p} & \mathcal{O}_{\p}
  \end{array} \right) : \dete A\in K_{\p}^\times, (a\mathcal{O}_{\p},\cn\mathcal{O}_{\p})=1  \},$$
  
  $$W_{\p}(\cn)=\{ x\in Y_{\p}(\cn) : \dete x\in \mathcal{O}^\times_{\p}  \}$$
  and put $$Y=Y(\cn)=G_A\cap \left(G_{\infty +}\times \prod_{\p} Y_{\p}(\cn)\right),$$
  $$W=W(\cn)=G_{\infty +}\times \prod_{\p} W_{\p}(\cn).$$

Given a numerical character $\phi$ modulo $\cn$ define a homomorphism $\phi_Y: Y\rightarrow \mathbb{C}^\times$ by setting $\phi_Y(\left(
  \begin{array}{ c c }
    \tilde{a}& * \\
      *&*
  \end{array} \right))=\phi(\tilde{a}_{\cn}\mbox{ mod }\cn )$.
  
Given a fractional ideal $\mathcal I$ of $K$ define $\tilde{\mathcal{I}}=(\mathcal{I}_{\nu})_{\nu}$ to be a fixed idele such that $\mathcal{I}_{\infty}=1$ and $\tilde{\mathcal{I}}\mathcal{O}=\mathcal{I}$. For $\lambda=1,...,h,$ set $x_{\lambda}=\left(
  \begin{array}{ c c }
1& 0 \\
      0&\tilde{I}_{\lambda}
  \end{array} \right)\in G_A$. By the Strong Approximation theorem we have $$G_A=\bigcup_{\lambda=1}^h G_K x_{\lambda} W=\bigcup_{\lambda=1}^h G_K x_{\lambda}^{-\iota} W$$ where $\iota$ denotes the canonical involution on two-by-two matrices.
  
  For an $h$-tuple $(f_1,...,f_h)\in\mathfrak{M}_k(\cn,\phi,\theta)$ we define a function $\f: G_A\rightarrow \mathbb{C}$ by
  
  $$\f(\alpha x_{\lambda}^{-\iota}w)=\phi_Y(w^{\iota})\dete(w_{\infty})^{im}(f_{\lambda}\mid w_{\infty})(\textbf{i})$$
for $\alpha\in G_K$, $w\in W(\cn)$ and $\textbf{i}=(i,...,i)$ (with $i=\sqrt{-1}$). Here $$f_{\lambda}\mid \left(
  \begin{array}{ c c }
a& b \\
      c&d
  \end{array} \right)(\tau)=(ad-bc)^{\frac{k}{2}}(c\tau+d)^{-k} f_{\lambda}\left(\frac{a\tau+b}{c\tau+d}\right).$$

As in \cite{shimura-ann, shimura-duke}, we identify $\mathfrak{M}_k(\cn,\phi,\theta)$ with the set of functions $\f: G_A\rightarrow \mathbb{C}$ satisfying
\begin{enumerate}
\item $\f(\alpha x w)=\phi_Y(w^{\iota})\f(x)$ for all $\alpha\in G_K, x\in G_A, w\in W(\cn), w_{\infty}=1$
\item For each $\lambda$ there exists an element $f_{\lambda}\in M_k$ such that $$\f(x_{\lambda}^{-\iota}y)=\dete(y)^{im}(f_{\lambda}\mid y)(\textbf{i})$$ for all $y\in G_{\infty +}$.
\end{enumerate}  
 
Let $\phi_{\infty}: K_A^\times\rightarrow \mathbb{C}^\times$ be defined by $\phi_{\infty}(\tilde{a})=\mbox{sgn}(\tilde{a}_{\infty})^k|\tilde{a}_{\infty}|^{2im}$, where $m$ was defined in the definition of $\theta$. We say that a Hecke character $\Phi$ extends $\phi\phi_{\infty}$ if $\Phi(\tilde{a})=\phi(\tilde{a}_{\cn}\mbox{ mod }\cn)\phi_{\infty}(\tilde{a})$ for all $\tilde{a}\in K_{\infty}^\times \times \prod_{\p} \mathcal{O}_{\p}^\times$. If $\mathfrak{P}_{\infty}$ denotes the $K$-modulus consisting of the product of all the infinite primes of $K$, then any Hecke character $\Phi$ extending $\phi\phi_{\infty}$ has conductor dividing $\cn\mathfrak{P}_{\infty}$. Henceforth we will use the word conductor to refer to the finite part of the conductor. 

If $\phi$ is a numerical character modulo $\cp\cn_0$ where $\cp=\p^a$ is a power of a prime $\p$ and $(\p,\cn_0)=1$, then by the Chinese Remainder Theorem we have a decomposition $\phi=\phi_{\cp}\phi_{\cn_0}$ where $\phi_{\cp}$ is a numerical character modulo $\cp$ and $\phi_{\cn_0}$ is a numerical character modulo $\cn_0$. If $\Phi_{\cp}$ is a Hecke character extending $\phi_{\cp}$ (i.e. trivial infinite part) and $\Phi_{\cn_0}$ is a Hecke character extending $\phi_{\cn_0}\phi_{\infty}$ then it is clear that $\Phi=\Phi_{\cp}\Phi_{\cn_0}$. Throughout this paper we shall adopt this convention and decompose Hecke characters $\Phi$ extending numerical characters modulo $\cp\cn_0$ as $\Phi=\Phi_{\cp}\Phi_{\cn_0}$ where $\Phi_{\cp}$ has trivial infinite part.

Given a Hecke character $\Phi$ extending $\phi\phi_{\infty}$ we define an ideal character $\Phi^*$ modulo $\cn\mathfrak{P}_{\infty}$ by 
\begin{displaymath}
\left\{ \begin{array}{ll}
\Phi^*(\p)=\Phi(\tilde{\pi}_{\p}) & \textrm{for $\p\nmid \cn$ and $\tilde{\pi}\mathcal{O}=\p,$}\\
\Phi^*(\mathfrak{a})=0 & \textrm{if $(\mathfrak{a},\cn)\neq 1$ }\\
\end{array} \right.
\end{displaymath}

Observe that for any $\tilde{a}\in K_A^\times$ with $(\tilde{a}\mathcal{O},\cn)=1$, $\Phi(\tilde{a})=\Phi^*(\tilde{a}\mathcal{O})\phi(\tilde{a}_{\cn})\phi_{\infty}(\tilde{a})$.

For $\tilde{s}\in K_A^\times$, define $\f^{\tilde{s}}(x)=\f(\tilde{s}x)$. The map $\tilde{s}\longrightarrow \left( \f\mapsto \f^{\tilde{s}}\right)$ defines a unitary representation of $K_A^\times$ in $\mathfrak{M}_k(\cn,\phi,\theta)$. By Schur's Lemma the irreducible subrepresentations are all one-dimensional (since $K_A^\times$ is abelian). For a character $\Phi$ on $K_A^\times$, let $\mathscr{M}_k(\cn,\Phi)$ denote the subspace of $\mathfrak{M}_k(\cn,\phi,\theta)$ consisting of all $\f$ for which $\f^{\tilde{s}}=\Phi(\tilde{s})\f$ and let $\mathscr{S}_k(\cn,\Phi)\subset \mathscr{M}_k(\cn,\Phi)$ denote the subspace of cusp forms. If $s\in K^\times$ then $\f^{s}=\f$. It follows that $\mathscr{M}_k(\cn,\Phi)$ is nonempty only when $\Phi$ is a Hecke character.
 
If $\f=(f_1,...,f_h)\in \mathfrak{M}_k(\cn,\phi,\theta)$, then each $f_{\lambda}$ has a Fourier expansion
$$f_{\lambda}(\tau)=a_{\lambda}(0)+\sum_{0\ll \xi\in\mathcal{I}_{\lambda}} a_{\lambda}(\xi) e^{2\pi i \mbox{tr} (\xi\tau)}.$$

If $\mathfrak{m}$ is an integral ideal then following Shimura we define the $\mathfrak{m}$-th `Fourier' coefficient of $\f$ by 
\begin{displaymath}
C(\mathfrak{m},\f)=\left\{ \begin{array}{ll}
N(\mathfrak{m})^{\frac{k_0}{2}}a_{\lambda}(\xi)\xi^{-\frac{k}{2}-im}& \textrm{if $\mathfrak{m}=\xi\mathcal{I}_{\lambda}^{-1}\subset\mathcal{O}$}\\
0 & \textrm{otherwise}\\
\end{array} \right.
\end{displaymath}
where $k_0=\mbox{max}\{k_1,...,k_n\}$.

Given $\f\in\mathfrak{M}_k(\cn,\phi,\theta)$ and $y\in G_A$ define a slash operator by setting $(\f\mid y)(x)=\f(xy^{\iota})$. 

For an integral ideal $\mathfrak{r}$ define the shift operator $B_{\mathfrak{r}}$ by $$\f\mid B_{\mathfrak{r}}=N(\mathfrak{r})^{-\frac{k_0}{2}} \f\mid \left(
  \begin{array}{ c c }
1& 0 \\
      0&\tilde{\mathfrak{r}}^{-1}
  \end{array} \right).$$
  The shift operator maps $\mathscr{M}_k(\cn,\Phi)$ to $\mathscr{M}_k(\mathfrak{r}\cn,\Phi)$ and takes cusp forms to cusp forms. Further, $C(\mathfrak{m},\f\mid B_{\mathfrak{r}})=C(\mathfrak{m}\mathfrak{r}^{-1},\f)$. It is clear that $\f \mid B_{\mathfrak{r}_1}\mid B_{\mathfrak{r}_2}=\f\mid B_{\mathfrak{r}_1\mathfrak{r}_2}$.
  
  For an integral ideal $\mathfrak{r}$ the Hecke operator $T_{\mathfrak{r}}=T_{\mathfrak{r}}^{\cn}$ maps  $\mathscr{M}_k(\cn,\Phi)$ to itself regardless of whether or not $(\mathfrak{r},\cn)=1$. This action is given on Fourier coefficients by $$C(\mathfrak{m},\f\mid T_{\mathfrak{r}})=\sum_{\mathfrak{m}+\mathfrak{r}\subset\mathfrak{a}} \Phi^*(\mathfrak{a})N(\mathfrak{a})^{k_0-1}C(\mathfrak{a}^{-2}\mathfrak{m}\mathfrak{r},\f).$$ Like the shift operator, $T_{\mathfrak{r}}$ takes cusp forms to cusp forms. Also note that if $(\mathfrak{a},\mathfrak{r})=1$ then $B_{\mathfrak{a}}T_{\mathfrak{r}}=T_{\mathfrak{r}}B_{\mathfrak{a}}$. Given $\f\in  \mathscr{S}_k(\cn,\Phi)$ we define the annihilator operator $A_{\p}$ by $$\f\mid A_{\p} = \f-\f\mid T_{\p}\mid B_{\p}.$$

Let $\mathscr{S}_k^-(\cn,\Phi)$ be the subspace of $\mathscr{S}_k(\cn,\Phi)$ generated by all $\g\mid B_{\mathcal{Q}}$ where $\g\in \mathscr{S}_k(\cn^\prime,\Phi)$ for some proper divisor $\cn^\prime$ of $\cn$ with $\mathcal{Q}\cn^\prime\mid \cn$. This space is invariant under the action of the Hecke operators $T_\mathfrak{r}$ with $(\mathfrak{r},\cn)=1$. 

Shimura defines ((2.28) of \cite{shimura-duke}) a Petersson inner product $\langle \f,\g\rangle$ for $\f,\g\in\mathscr{S}_k(\cn,\Phi)$. With respect to this inner product the Hecke operators satisfy $$\Phi^*(\mathfrak{m})\langle\f\mid T_{\mathfrak{m}},\g\rangle=\langle \f,\g\mid T_{\mathfrak{m}}\rangle$$ for integral ideals $\mathfrak{m}$ coprime to $\cn$. Let $\mathscr{S}_k^+(\cn,\Phi)$ denote the orthogonal complement of $\mathscr{S}_k^-(\cn,\Phi)$ in $\mathscr{S}_k(\cn,\Phi)$. It follows from our discussion above that $\mathscr{S}_k^+(\cn,\Phi)$ is invariant under the Hecke operators $T_{\mathfrak{r}}$ with $(\mathfrak{r},\cn)=1$.

\begin{definition}A newform $\f$ in $\mathscr{S}_k(\cn,\Phi)$ is a form in $\mathscr{S}_k^+(\cn,\Phi)$ which is a simultaneous eigenform for all Hecke operators $T_{\q}$ with $\q$ a prime not dividing $\cn$. We say that $\f$ is normalized if $C(\mathcal{O},\f)=1$.\end{definition}

As in the classical case, if $\f\in  \mathscr{S}_k(\cn,\Phi)$ is a newform with Hecke eigenvalues $\{\lambda_{\p} : \p \mbox{is prime} \},$ then $C(\p,\f)=\lambda_{\p}C(\mathcal{O},\f)$ for all primes $\p\nmid\cn$.

Since $\{ T_{\q} : \q\nmid \cn\}$ is  commuting family of hermitian operators,  $\mathscr{S}_k^+(\cn,\Phi)$ has an orthogonal basis consisting of newforms. If $\g\in \mathscr{S}_k^-(\cn,\Phi)$ is a simultaneous eigenform for all $T_{\q}$ with $\q\nmid \cn$ then there exists a newform $\textbf{h}\in  \mathscr{S}_k^+(\cn^\prime,\Phi)$ with $\cn^\prime\mid \cn$ having the same eigenvalues as $\g$ for all such $T_{\q}$.
  
Finally, if $\f, \g\in \mathscr{S}_k(\cn,\Phi)$ are both simultaneous eigenforms for all Hecke operators $T_{\q}$ with $\q$ a prime not dividing $\cn$ having the same Hecke eigenvalues, then we say that $\f$ is equivalent to $\g$ and write $\f \sim \g$. If $\f$ is a newform and $\f\sim \g$, then there exists $c\in \mathbb{C}^\times$ such that $\f=c\g$. This follows from Theorem 3.5 of \cite{shemanske-walling}.
%%%%%%%%%%%%%%%%%%%%%
%%%%%%%%%%%%%%%%%%%%%
%%%%%%%%%%%%%%%%%%%%%
%%%%%%%%%%%%%%%%%%%%%
%%%%%%%%%%%%%%%%%%%%%
%%%%%%%%%%%%%%%%%%%%%
%%%%%%%%%%%%%%%%%%%%%
%%%%%%%%%%%%%%%%%%%%%

\section{Twists of Newforms}

Throughout this section $\p$ will denote a fixed prime ideal of $\mathcal{O}$.

Fix an integral ideal $\cn$ and write $\cn=\cp{\mathcal{N}_0}$ where $\cp$ is the $\p$-primary part of $\cn$ and $(\cp,{\mathcal{N}_0})=1$.

Fix a space $\mathscr{S}_k(\cn,\Phi)\subset \mathfrak{S}_k(\cn,\phi, m)$, where $\Phi$ is a Hecke character extending $\phi\phi_{\infty}$.

\begin{definition} If $\f\in\mathscr{S}_k(\cn,\Phi)$ is a normalized newform and $\Psi$ is a Hecke character then we define the twist of $\f$ by $\Psi$, denoted $\f_\Psi$, by 

$$\textbf{f}_\Psi(x)=\tau(\overline{\Psi})^{-1}\Psi(\dete x) \sum_{r\in \mathfrak{f}_{\Psi}^{-1}\mathfrak{d}^{-1}/\mathfrak{d}^{-1}} \overline{\Psi}_{\infty}(r)\overline{\Psi}^*(r\mathfrak{f}_{\Psi}\diff)\f\mid \left(
  \begin{smallmatrix}
    1& r \\
      0& 1
  \end{smallmatrix} \right)_0 (x),
$$ where $\tau(\overline{\Psi})$ is the Gauss sum associated to $\overline{\Psi}$ defined in (9.31) of \cite{shimura-ann} and the subscript $0$ denotes the projection onto the nonarchimedean part.
\end{definition}

\begin{proposition}\label{proposition:crudebound}
Let notation be as above and set $\mathcal{L}=lcm\{\cn,\cond_{\Phi}\mathfrak{f}_{\Psi},\mathfrak{f}_{\Psi}^2\}$. If $\f\in\mathscr{S}_k(\cn,\Phi)$ is a normalized newform then $\f_{\Psi}\in \mathscr{S}_k(\mathcal{L},\Psi^2\Phi)$ and $C(\mathfrak{m},\f_{\Psi})=\Psi^*(\mathfrak{m})C(\mathfrak{m},\f)$ for all integral ideals $\mathfrak{m}$.
\end{proposition}
\begin{proof}
This is Proposition 4.5 of \cite{shimura-duke}.
\end{proof}

The following proposition is trivial to verify using the action of the Hecke operators on Fourier coefficients.

\begin{proposition}\label{proposition:eigentwist}
Let notation be as above and $\mathfrak q$ be a prime with $\mathfrak{q}\nmid \mathfrak{f}_{\Psi}$. For $\f\in \mathscr{S}_k(\mathcal{N},\Phi)$ we have $\f_{\Psi}\mid T_{\mathfrak{q}} = \Psi^*(\mathfrak q) (\f\mid T_{\mathfrak q})_{\Psi}$.
\end{proposition}

Although Proposition \ref{proposition:crudebound} gives an upper bound for the exact level of $\f_{\Psi}$, one can obtain better bounds in certain special cases. Of particular interest to us is the case in which $\Psi=\overline{\Phi}_{\cp}$. The following proposition gives an improved bound on the level of $\f_{\Psi}$ in this special case and generalizes Proposition 3.6 of \cite{atkin-li}.

\begin{proposition}\label{proposition:levelbound}
Let $\mathfrak{f}$ be the conductor of $\Phi_{\mathcal{P}}$. Set

\begin{displaymath}
\cl = \left\{ \begin{array}{ll}
\mathcal{N} & \textrm{if $ord_{\mathfrak{p}}(\mathfrak{f})<ord_{\mathfrak{p}}(\mathcal{P})$}\\
\mathfrak{p}\mathcal{N} & \textrm{if $ord_{\mathfrak{p}}(\mathfrak{f})=ord_{\mathfrak{p}}(\mathcal{P})$}\\
\end{array} \right.
\end{displaymath}
If $\f\in \mathscr{S}_k(\cn,\Phi)$ then $\textbf{f}_{\overline{\Phi}_{\mathcal{P}}}\in \mathscr{S}_k(\cl,\overline{\Phi}_{\mathcal{P}}\Phi_{{\mathcal{N}_0}})$.
\end{proposition}

\begin{proof}

Let $\alpha\in G_K, x\in G_A$ and $w\in W(\cl)$ with $w_\infty=1$. We will show that $$\textbf{f}_{\overline{\Phi}_{\mathcal{P}}}(\alpha x w)=(\phi_{{\mathcal{N}_0}}\overline{\phi}_{\cp})_Y(w^\iota) \textbf{f}_{\overline{\Phi}_{\mathcal{P}}}(x).$$ Write $w=\left(
  \begin{array}{c c}
     \tilde{a} & \td^{-1}\tilde{b} \\
      \tilde{c}\tl\td& \tilde{d}
  \end{array} \right)$.

Let $r \in \mathfrak{f}^{-1}\mathfrak{d}^{-1}/\mathfrak{d}^{-1}$ and observe that by the Strong Approximation theorem there exists an element $r^\prime\in K$ such that

\begin{enumerate}
\item $ord_\mathfrak{q}(r^\prime)\geq 0$ for all primes $\mathfrak{q}$ such that $ord_{\q}(\cond\diff)=0$
\item $ord_\mathfrak{q}(r^\prime)\geq -ord_{\mathfrak{q}}(\mathfrak d)$ for all primes $\mathfrak q$ such that $q\mid \mathfrak d$ and $\mathfrak{q}\neq \mathfrak{p}$
\item $a_{\mathfrak{p}}r-r^\prime(d_{\mathfrak{p}}-c_{\mathfrak{p}}\cl_{\p} \mathfrak{d}_{\p} r) \in \diff^{-1}\mathcal{O}_{\p}$
\end{enumerate}
Note that such an $r^\prime$ lies in $\mathfrak{f}^{-1}\mathfrak{d}^{-1}$. We claim that such a $r^\prime$ is uniquely determined in $\mathfrak{f}^{-1}\mathfrak{d}^{-1}/\mathfrak{d}^{-1}$. More precisely, suppose that $r_0, r_1 \in \mathfrak{f}^{-1}\mathfrak{d}^{-1}/\mathfrak{d}^{-1}$ give rise to $r^0,r^1\in\mathfrak{f}^{-1}\mathfrak{d}^{-1}/\mathfrak{d}^{-1}$. We will show that if $r^0+\diffinv=r^1+\diffinv$ then $r_0+\diffinv=r_1+\diffinv$. To do this we will suppose that $(r^0-r^1)\in\diffinv$ and show that $(r_0-r_1)\in\diffinv\mathcal{O}_{\mathfrak{q}}$ for all finite primes $\q$. It will then follow from the local-global correspondence for lattices that $(r_0-r_1)\in\diffinv$.

We have two cases to consider.

Case 1 - $\q\neq\p$: Both $r_0$ and $r_1$ lie in $\cond^{-1}\diffinv$ and hence in $\cond^{-1}\diffinv\mathcal{O}_{\q}=\mathcal{O}_{{\q}}^\times\diffinv\mathcal{O}_{\q}\subset \diffinv\mathcal{O}_{\q}$. It follows that $(r_0-r_1)\in\diffinv\mathcal{O}_{\q}$.

Case 2 - $\q=\p$: By condition (3) we have

$$r_0a_{\p}-r^0(d_{\p}-c_{\p}\mathcal{L}_{\p} \mathfrak{d}_{\p}r_0)\in\mathfrak{d}^{-1}\mathcal{O}_{\p}$$
and 
$$r_1a_{\p}-r^1(d_{\p}-c_{\p}\mathcal{L}_{\p} \mathfrak{d}_{\p}r_1)\in\mathfrak{d}^{-1}\mathcal{O}_{\p}.$$

Putting these together yields 
$$r_0(a_{\p})-r^0(d_{\p}-c_{\p}\mathcal{L}_{\p} \mathfrak{d}_{\p}r_0)-r_1(a_{\p})+r^1(d_{\p}-c_{\p}\mathcal{L}_{\p} \mathfrak{d}_{\p}r_1)\in\mathfrak{d}^{-1}\mathcal{O}_{\p}.$$

Observe that by definition of $\cl$ and $W(\cl)$, each of the terms in parentheses lies in $\mathcal{O}_{\p}^\times$. We may therefore ease notation by writing $u_i$ for the parenthesized unit:
\begin{equation}\label{equation:eq1}
r_0 u_1 - r^0 u_2 - r_1 u_3 + r^1 u_4\in\mathfrak{d}^{-1}\mathcal{O}_{\p}.
\end{equation}
Also observe that $$u_1+\diffinv\mathcal{O}_{\p}=a_{\p}+\diffinv\mathcal{O}_{\p}=u_3+\diffinv\mathcal{O}_{\p}$$ and $$u_2+\diffinv\mathcal{O}_{\p}=d_{\p}+\diffinv\mathcal{O}_{\p}=u_4+\diffinv\mathcal{O}_{\p}.$$

It follows that $$r_0 u_1+\diffinv\mathcal{O}_{\p}=r_0 a_{\p}+\diffinv\mathcal{O}_{\p},$$
$$r_1 u_3+\diffinv\mathcal{O}_{\p}=r_1 a_{\p}+\diffinv\mathcal{O}_{\p},$$
$$r^0 u_2+\diffinv\mathcal{O}_{\p}=r^0 d_{\p}+\diffinv\mathcal{O}_{\p},$$ and
$$r^1 u_4+\diffinv\mathcal{O}_{\p}=r^1 d_{\p}+\diffinv\mathcal{O}_{\p}.$$

Suppose that $(r^0-r^1)\in\diffinv\subset\diffinv\mathcal{O}_{\p}$. Then $d_{\p}(r^0-r^1)\in\diffinv\mathcal{O}_{\p}$ as well. We conclude, by Equation \ref{equation:eq1}, that $(r_1 u_3-r_0 u_1)\in\diffinv\mathcal{O}_{\p}$. This means that $a_{\p}(r_1-r_0)\in\diffinv\mathcal{O}_{\p}$, hence $(r_1-r_0)\in\diffinv\mathcal{O}_{\p}$.

We have shown that $(r^0-r^1)\in\diffinv$ implies that $(r_0-r_1)\in\diffinv\mathcal{O}_{\q}$ for all finite primes $\q$, hence $(r_0-r_1)\in\diffinv$.

We now show that $$\textbf{f}_{\overline{\Phi}_{\mathcal{P}}}(\alpha x w)=(\phi_{{\mathcal{N}_0}}\overline{\phi}_{\cp})_Y(w^\iota) \textbf{f}_{\overline{\Phi}_{\mathcal{P}}}(x).$$

By definition,

\begin{equation}\label{equation:line1}
\textbf{f}_{\overline{\Phi}_{\mathcal{P}}}(\alpha x w)=\tau(\Phi_{\mathcal{P}})^{-1}\overline{\Phi}_{\mathcal{P}}(\dete(\alpha x w))\displaystyle\sum _{r\in \cond^{-1}\diffinv/\diffinv} \Phi_{\mathcal{P}}^*(r\cond\diff)\f\mid \left(
  \begin{smallmatrix}
    1& r \\
      0& 1
  \end{smallmatrix} \right)_0 (\alpha x w)
  \end{equation}

  \begin{equation}\label{equation:line2}
  =\tau(\Phi_{\mathcal{P}})^{-1}\overline{\Phi}_{\mathcal{P}}(\dete(x))\overline{\Phi}_{\mathcal{P}}(\dete(w))\displaystyle\sum _{r\in \cond^{-1}\diffinv/\diffinv}\Phi_{\mathcal{P}}^*(r\cond\diff)\f(\alpha x w  \left(
  \begin{smallmatrix}
    1& -r \\
      0& 1
  \end{smallmatrix} \right)_0 )\end{equation}
  
  Let $r^\prime\in \cond^{-1}\diffinv / \diffinv$ correspond to $r$ (i.e. $r^\prime$ satisfies the three conditions listed in the first paragraph of this proof) and $w^\prime$ be a solution to the matrix equation

  \begin{equation}\label{equation:matrixeq}
  w \left(
  \begin{array}{ c c }
    1& -r \\
      0& 1
  \end{array} \right)_0=   \left(
  \begin{array}{ c c }
    1& -r^\prime \\
      0& 1
  \end{array} \right)_0w^\prime.
  \end{equation}
  
  We note that 
  
  \begin{equation}\label{equation:matrixwprime}
  w^\prime=\left(
  \begin{array}{ c c }
    \tilde{a} + \tilde{c} \tl \td r^\prime& \td^{-1}\tilde{b}+\tilde{d}r^\prime-\tilde{a}\nu-\tilde{c}\tl\td r r^\prime \\
      \tilde{c}\tl\td& \tilde{d}-\tilde{c}\tl\td r
  \end{array} \right)
  \end{equation}
  
  and that the three conditions defining $r^\prime$ imply that $w^\prime\in W(\cn)$.

Substituting equation \ref{equation:matrixeq} into equation \ref{equation:line2} yields

\begin{equation}\label{equation:line3}
 \tau(\Phi_{\mathcal{P}})^{-1}\overline{\Phi}_{\mathcal{P}}(\dete(x))\overline{\Phi}_{\mathcal{P}}(\dete(w))\displaystyle\sum _{r\in \cond^{-1}\diffinv/\diffinv} \Phi_{\mathcal{P}}^*(r\cond\diff)\f(\alpha x \left(
  \begin{smallmatrix}
    1& -r^\prime \\
      0& 1
  \end{smallmatrix} \right)_0 w^\prime )
\end{equation}

Because $\f\in \mathscr{S}_k(\mathcal{N},\Phi)$, we may rewrite this as

\begin{equation}\label{equation:line4}
 =\tau(\Phi_{\mathcal{P}})^{-1}\overline{\Phi}_{\mathcal{P}}(\dete(x))\overline{\Phi}_{\mathcal{P}}(\dete(w))\displaystyle\sum _{r\in \cond^{-1}\diffinv/\diffinv} \Phi_{\mathcal{P}}^*(r\cond\diff){\phi}_Y ( ({w^\prime})^\iota )\f\mid \left(
  \begin{smallmatrix}
    1& r^\prime \\
      0& 1
  \end{smallmatrix} \right)_0 (x)  
\end{equation}

\begin{equation}\label{equation:line5}
 =\tau(\Phi_{\cp})^{-1}\overline{\Phi}_{\cp}(\dete(x))\overline{\Phi}_{\mathcal{P}}(\dete(w))\displaystyle\sum _{r\in \cond^{-1}\diffinv/\diffinv}\Phi_{\cp}^*(r\cond\diff)\phi_{{\mathcal{N}_0}}(d_{\p})\phi_{\cp}( d_{\p})\f\mid \left(
  \begin{smallmatrix}
    1& r^\prime \\
      0& 1
  \end{smallmatrix} \right)_0 (x)  
\end{equation}

\begin{equation}\label{equation:line5}
 =\phi_{{\mathcal{N}_0}}(d_{\p})\tau(\Phi_{\cp})^{-1}\overline{\Phi}_{\cp}(\dete(x))\overline{\Phi}_{\mathcal{P}}(\dete(w))\displaystyle\sum _{r\in \cond^{-1}\diffinv/\diffinv} \Phi_{\cp}^*(r\cond\diff)\phi_{\cp}( d_{\p})\f\mid \left(
  \begin{smallmatrix}
    1& r^\prime \\
      0& 1
  \end{smallmatrix} \right)_0 (x)  
\end{equation}

We proceed by rewriting the sum in terms of $r^\prime$ rather than $r$. To do this we consider the expression
$$\Phi_{\cp}^*(r\cond\diff)\phi_{\cp}( d_{\p})$$
inside the summation.

As $\Phi_{\cp}(\tilde{\alpha})=\Phi_{\cp}^*(\tilde{\alpha}\mathcal{O}_K)\phi_{\cp}(\tilde{\alpha})$ for all $\tilde{\alpha}\in J_K$ with $(\tilde{\alpha}\mathcal{O}_K,\p)=1$ this expression is equal to:

$$\Phi_{\cp}(r\tf\td)\overline{\phi}_{\cp}(r\mathfrak{f}_{\p}\mathfrak{d}_{\p})\phi_{\cp}(d_{\p}).$$

Setting $D_{\p}=\dete(w_{\p})=a_{\p}d_{\p}$, we rewrite this expression as

$$\Phi_{\cp}(r\tf\td)\overline{\phi}_{\cp}(r\mathfrak{f}_{\p}\mathfrak{d}_{\p})\overline{\phi}_{\cp}(a_{\p} D^{-1}_{\p})=\Phi_{\cp}(r\tf\td)\overline{\phi}_{\cp}(a_{\p}r\mathfrak{f}_{\p}\mathfrak{d}_{\p})\overline{\phi}_{\cp}( D^{-1}_{\p}).$$

Recall the third condition defining $r^\prime$: $a_{\mathfrak{p}}r-r^\prime(d_{\mathfrak{p}}-c_{\mathfrak{p}}\mathcal{L}_{\p}\mathfrak{d}_{\p} r) \in \diff^{-1}\mathcal{O}_{\p}$. This implies $$a_{\mathfrak{p}}r\mathfrak{f}_{\p}\mathfrak{d}_{\p}-r^\prime\mathfrak{f}_{\p}\mathfrak{d}_{\p}(d_{\mathfrak{p}}-c_{\mathfrak{p}}\tilde{\cl}_{\p}\mathfrak{d}_{\p} r) \in\cond\mathcal{O}_{\p},$$ and in particular, $a_{\mathfrak{p}}r\mathfrak{f}_{\p}\mathfrak{d}_{\p}-d_{\mathfrak{p}}r^\prime\mathfrak{f}_{\p}\mathfrak{d}_{\p}\in\cond\mathcal{O}_{\p}$. This, along with the fact that $\Phi_{\cp}(r)=\Phi_{\cp}(r^\prime)=1$, shows that we now have

$$\Phi_{\cp}(r^\prime\tf\td)\overline{\phi}_{\cp}(d_{\p}r^\prime\mathfrak{f}_{\p}\mathfrak{d}_{\p})\overline{\phi}_{\cp}( D^{-1}_{\p})=\Phi_{\cp}(r^\prime\tf\td)\overline{\phi}_{\cp}(r^\prime\mathfrak{f}_{\p}\mathfrak{d}_{\p})\phi_{\cp}(a_{\p})=\Phi_{\cp}^*(r^\prime\cond\diff)\phi_{\cp}(a_{\p}).$$

We have shown that $\Phi_{\cp}^*(r\cond\diff)\phi_{\cp}( d_{\p})=\Phi_{\cp}^*(r^\prime\cond\diff)\phi_{\cp}(a_{\p})$.

We rewrite equation \ref{equation:line5} as

\begin{equation}\label{equation:line6}
 =\phi_{{\mathcal{N}_0}}(d_p)\tau(\Phi_{\cp})^{-1}\overline{\Phi}_{\cp}(\dete(x))\overline{\Phi}_{\mathcal{P}}(\dete(w))\phi_{\cp}( a_{\p})\displaystyle\sum _{r^\prime \in \cond^{-1}\diffinv/\diffinv} \Phi_{\cp}^*(r^\prime\cond\diff)\f\mid \left(
  \begin{smallmatrix}
    1& r^\prime \\
      0& 1
  \end{smallmatrix} \right)_0 (x)  
\end{equation}

By definition of $W(\cl)$, $\dete(w)\in\mathcal{O}_{\q}^\times$ for all finite primes $\q$. It follows that $$\overline{\Phi}_{\cp}(\dete(w))=\overline{\phi}_{\cp}(\dete(w))=\overline{\phi}_{\cp}(a_{\p}d_{\p}).$$

We therefore rewrite equation \ref{equation:line6} as

\begin{equation}\label{equation:line7}
 =\phi_{{\mathcal{N}_0}}(d_p)\tau(\Phi_{\cp})^{-1}\overline{\Phi}_{\cp}(\dete(x))\overline{\phi}_{\cp}(a_{\p}d_{\p})\phi_{\cp}( a_{\p})\displaystyle\sum _{r^\prime \in \cond^{-1}\diffinv/\diffinv} \Phi_{\cp}^*(r^\prime\cond\diff)\f\mid \left(
  \begin{smallmatrix}
    1& r^\prime \\
      0& 1
  \end{smallmatrix} \right)_0 (x)  
\end{equation}

This is equal to $\phi_{{\mathcal{N}_0}}(d_{\p})\overline{\phi}_{\cp}(d_{\p})\f_{\overline{\Phi}_{\mathcal{P}}}(x)=(\phi_{{\mathcal{N}_0}}\overline{\phi}_{\cp})_Y(w^\iota)\f_{\overline{\Phi}_{\mathcal{P}}}(x)$.

Therefore $\f_{\overline{\Phi}_{\mathcal{P}}}(\alpha x w)=(\phi_{{\mathcal{N}_0}}\overline{\phi}_{\cp})_Y(w^\iota)\f_{\overline{\Phi}_{\mathcal{P}}}(x)$ for $\alpha\in G_K, x\in G_A$ and $w\in W(\mathcal{L})$ with $w_\infty=1$. 

It follows that $\f_{\overline{\Phi}_{\mathcal{P}}}\in \mathscr{S}_k(\mathcal{L},\overline{\Phi}_{\mathcal{P}}\Phi_{{\mathcal{N}_0}})$.\end{proof}

If $\f\in \mathscr{S}_k^+(\cn,\Phi)$ is a normalized newform and $\Psi$ is a Hecke character with $(\mathfrak{f}_{\Phi},\mathfrak{f}_{\Psi})=1$, then $\f_{\Psi}$ is always a normalized newform of $\mathscr{S}_k^+(\mathfrak{f}_{\Psi}^2\cn,\Psi^2\Phi)$ by Theorem 5.5 of \cite{shemanske-walling}. The situation when the conductors of $\Phi$ and $\Psi$ are not coprime is much more subtle and will be studied throughout the remainder of this paper. Clearly it suffices to consider characters whose conductor is a power of a single prime dividing the level $\cn$. We therefore suppose that $\Psi$ is a $\p$-primary Hecke character. 

Henceforth we assume that $\Psi$ is a Hecke character with conductor dividing $\cp$. The infinite part of $\Psi$ has the form $\Psi_{\infty}(a)=\mbox{sgn}(a)^l |a|^{ir}$ for $l\in\mathbb Z^n$, $r\in\mathbb R^n$ and $a\in K_{\infty}^\times$. In what follows we shall always choose $\Psi$ so that $r=0$.

We will see that the vanishing of $C(\p,\f)$ lies at the heart of the question of whether or not $\f_{\Psi}$ is a newform of $\mathscr{S}_k(\cn,\Psi^2\Phi)$. We present a slightly strengthened version of Theorem 3.3 of \cite{shemanske-walling}, which will allow us to determine when $C(\p,\f)\neq 0$.

\begin{theorem}\label{theorem:threethree} Let $\f$ be a normalized newform lying in $\mathscr{S}_k(\cn,\Phi)$.
\begin{enumerate}
\item The Dirichlet series attached to $\f$, $D(s,f)=\sum_{\mathfrak{m}\subset \mathcal{O}} C(\mathfrak{m},\f) N(\mathfrak{m})^{-s}$ has an Euler product 

$$D(s,\f)=\prod_{\qz\mid \cn} (1-C(\qz,\f)N(\qz)^{-s})^{-1}\times \prod_{\qo\nmid\cn}(1-C(\qo,\f)N(\qo)^{-s}+\Phi^*(\qo)N(\qo)^{k_0-1-2s})^{-1}$$

\item If $\phi$ is not defined modulo $\cn \p^{-1}$, then $|C(\p,\f)|=N(\p)^{\frac{(k_0-1)}{2}}$.

\item If $\phi$ is a character modulo $\cn\p^{-1}$, then $C(\p,\f)=0$ if $\p^2\mid\cn$ and $|C(\p,\f)|^2=N(\p)^{k_0-2}$ if $\p^2\nmid\cn$.
\end{enumerate}
\end{theorem}

\begin{proof}
The statement of this theorem differs from Theorem 3.3 of \cite{shemanske-walling} only in that part 2 of the latter showed that either $C(\p,\f)=0$ or $|C(\p,\f)|=N(\p)^{\frac{(k_0-1)}{2}}$ and that $C(\p,\f)$ was non-zero for a set of primes having density 1. Kevin Buzzard has recently shown that in fact, $C(\p,\f)$ is never zero (see \cite{buzzard}), allowing us to state the above theorem in its strengthened form.
\end{proof}

Henceforth we use the letter $\nu$ to denote $ord_{\p}(\cp)=ord_{\p}(\cn)$.

\begin{lemma}\label{lemma:vanishingcoeffs}
Assume that $\nu \geq 2$ and that $e(\Phi_{\cp})<\nu$. If $\textbf{f}\in\mathscr{S}_k^+(\cn,\Phi)$ is a normalized newform then $\textbf{f}_{\overline{\Psi}\Psi}=\textbf{f}$. In particular, $$\mathscr{S}_k^+(\cn, \Phi)^{\overline{\Psi}\Psi}=\mathscr{S}_k^+(\cn, \Phi).$$
\end{lemma}
\begin{proof}
It follows immediately from Theorem \ref{theorem:threethree}.(3) that $C(\mathfrak{p},\textbf{f})=0$. Because $\f$ is an eigenform of $T_{\p}$ with eigenvalue $C(\p,\f)$, $C(\mathscr{I}\mathfrak{p},\textbf{f})=C(\mathscr{I},\textbf{f})C(\mathfrak{p},\textbf{f})=0$ for all integral ideals $\mathscr{I}$. Thus the annihilator operator $A_{\mathfrak p}$ acts as the identity operator on the newforms of level $\cn$ and character $\Phi$. The first part therefore follows from the observation that $\textbf{f}_{\overline{\Psi}\Psi}=\textbf{f}\mid A_\mathfrak{p}$.  As newforms generate the space $\mathscr{S}_k^+(\cn, \Phi)$, we have the second part as well.\end{proof}

\begin{proposition}\label{proposition:liesinnewformspace}
Assume that $\nu\geq 2$ and that $0<e(\Phi_{\cp})<\nu$. If $\textbf{f}\in \mathscr{S}_k^+(\cn,\Phi)$ is a newform then $\textbf{f}_{\overline{\Phi}_{\cp}} \in \mathscr{S}_k^+(\cn,\overline{\Phi}_{\cp}\Phi_{{\mathcal{N}_0}})$ is a newform as well.
\end{proposition}
\begin{proof}
Let  $\textbf{f}\in \mathscr{S}_k^+(\cn,\Phi)$ be a normalized newform. By  Proposition \ref{proposition:levelbound}, $\textbf{f}_{\overline{\Phi}_{\cp}} \in \mathscr{S}_k(\cn, \overline{\Phi}_{\cp}\Phi_{{\mathcal{N}_0}})$, and by Proposition \ref{proposition:eigentwist}, $\f_{\overline{\Phi}_{\cp}}$ is an eigenfunction of all the Hecke operators $T_\mathfrak{q}$ with $\mathfrak{q}$ a prime not dividing $\cn$, so there exists an ideal $\cn_0^\prime\mid\cn_0$, an integer $\mu$ satisfying $1\leq e(\Phi_{\cp})\leq\mu \leq \nu$ and a newform $\g\in \mathscr{S}_k^+(\p^\mu {\mathcal{N}_0^\prime},\overline{\Phi}_{\cp}\Phi_{{\mathcal{N}_0}})$ such that $\f_{\overline{\Phi}_{\cp}} \sim \g$.  We claim that $\cn_0^\prime=\cn_0$. Note that $\f=\f_{\overline{\Phi}_{\cp}\Phi_{\cp}}\sim \g_{\Phi_{\cp}}$ by Lemma \ref{lemma:vanishingcoeffs}, where $\g_{\Phi_{\cp}}$ has level $\p^\lambda \cn_0^\prime$ for some non-negative integer $\lambda$. Thus $\cn_0\mid\cn_0^\prime$, hence $\cn_0=\cn_0^\prime$.

If $\mu=\nu$ then $\f_{\overline{\Phi}_{\cp}}$ and $\g$ are of the same level, hence there exists $c\in\mathbb{C}$ such that $\f_{\overline{\Phi}_{\cp}}=c\g$. As both forms are normalized, $c=1$ and $\f_{\overline{\Phi}_{\cp}}=\g$ is a newform, finishing the proof. We may therefore suppose that $\mu<\nu$. We claim that  $e(\overline{\Phi}_{\cp})<\mu$. To show this, we will assume that $e(\overline{\Phi}_{\cp})=e(\Phi_{\cp})=\mu$ and derive a contradiction.

Because $\f_{\overline{\Phi}_{\cp}} \sim \g$, we have $\f_{\overline{\Phi}_{\cp}\Phi_{\cp}}\sim \g_{\Phi_{\cp}}$ as well. By Lemma \ref{lemma:vanishingcoeffs}, $\f_{{\overline{\Phi}_{\cp}\Phi_{\cp}}}=\f$, hence $\f\sim \g_{\Phi_{\cp}}$. By Proposition \ref{proposition:levelbound}, $\g_{\Phi_{\cp}}\in \mathscr{S}_k(\p^{\mu+1}{\mathcal{N}_0},\Phi)$. Therefore $\nu\leq \mu+1$, meaning that $$ \mu+1\geq \nu > \mu.$$

It is thus clear than $\nu=\mu+1$. This means that $\f$ is a newform of level $\p^{\mu+1}{\mathcal{N}_0}$ and character $\Phi$ and $\g_{\Phi_{\cp}}$ is a normalized cuspform in the same space which is equivalent to it. Therefore there exists $c\in\mathbb{C}$ such that $\f=c\g_{\Phi_{\cp}}$. As both $\f$ and $\g_{\Phi_{\cp}}$ are normalized, we see that $c=1$ and $\f=\g_{\Phi_{\cp}}$. But as $C(\p,\g)\neq 0$ by Theorem \ref{theorem:threethree}(2), this contradicts Corollary 6.4 of \cite{shemanske-walling}, which implies that $\g_{\Phi_{\cp}}$ is not a newform of any level.

We conclude that $e(\overline{\Phi}_{\cp})<\mu$. If $\mu\geq 2$ then Theorem \ref{theorem:threethree}(3) implies that the $\p$-th coefficient $C(\p,\g)$ of $\g$ is zero. Since $C(\p,\g)=0$ we have $\g=\g\mid A_{\p}$. But $$\textbf{f}_{\overline\Phi_{\cp}}=c_{\mathcal{O}}\g+c_{\p}\g\mid B_{\p}$$ and one easily checks by comparing Fourier coefficients that $c_{\mathcal{O}}=1$ and $c_{\p}=-C(\p,\g)$. Then $\textbf{f}_{\overline{\Phi}_{\cp}}=\g-C(\p,\g)\g\mid B_{\p}=\g\mid A_{\p}=\g$. Therefore $\f_{\overline{\Phi}_{\cp}}$ is a newform and we're done.

Now suppose that $\mu=1$. Then $e(\overline{\Phi}_{\cp})<\mu$ implies that $\Phi_{\cp}$ is trivial. This contradicts our hypothesis that $\Phi_{\cp}$ is nontrivial. \end{proof}

\begin{proposition}\label{proposition:liesinnewformspacearbitrary}
Assume that $0<e(\Psi)< \frac{\nu}{2}$ and $e(\Phi_{\cp})+e(\Psi)< \nu$. 

If $\f\in \mathscr{S}_k^+(\cn,\Phi)$ is a newform then $\f_\Psi\in \mathscr{S}_k^+(\cn,\Psi^2\Phi)$ is a newform as well.
\end{proposition}

\begin{proof}
We begin by noting that our hypotheses imply that $\nu\geq 3$. By Proposition \ref{proposition:crudebound}, $\f_\Psi\in \mathscr{S}_k(\p^\nu {\mathcal{N}_0},\Psi^2\Phi)$. Since $\f_\Psi$ is an eigenfunction of all the Hecke operators $T_{\q}$ with $\q$ a prime not dividing $\cn$ by Proposition \ref{proposition:eigentwist}, there exists an ideal $\cn_0^\prime\mid \cn_0$, an integer $\mu$ satisfying $0\leq e(\Phi_{\cp}\Psi^2)\leq\mu\leq \nu$ and a newform $\g\in \mathscr{S}_k^+(\p^\mu {\mathcal{N}_0^\prime},\Psi^2\Phi)$ such that $\f_\Psi \sim \g$. An argument identical to the one used in Proposition \ref{proposition:liesinnewformspace} shows that $\cn_0^\prime=\cn_0$. 

We will show that $e(\Phi_{\cp}\Psi^2)<\mu$ by assuming that $e(\Phi_{\cp}\Psi^2)=\mu$ and deriving a contradiction. Let $L=\mbox{max}\{\mu,e(\Phi_{\cp}\Psi^2)+e(\Psi),2e(\Psi)\}$. As  $\f_\Psi \sim \g$, we have, by Lemma \ref{lemma:vanishingcoeffs}, $\f=\f_{\Psi\overline{\Psi}}\sim \g_{\overline{\Psi}}$ where $\g_{\overline{\Psi}} \in \mathscr{S}_k(\p^L {\mathcal{N}_0},\Phi)$ by Proposition \ref{proposition:crudebound}. Therefore $L\geq \nu$. We have three cases to consider.

Case 1: $L=2e(\Psi)$. In this case $2e(\Psi)\geq \nu$ implies that $e(\Psi)\geq \frac{\nu}{2}$, contradicting our hypothesis that $e(\Psi)<\frac{\nu}{2}$.

Case 2: $L=e(\Phi_{\cp}\Psi^2)+e(\Psi)$. We have three subcases to consider. First suppose that $e(\Phi_{\cp})>e(\Psi)$. Then $e(\Phi_{\cp}\Psi^2)=e(\Phi_{\cp})$, hence $L\geq \nu$ implies that $e(\Phi_{\cp})+e(\Psi)\geq \nu$, contradicting our hypothesis that   $e(\Phi_{\cp})+e(\Psi)<\nu$. If $e(\Psi)>e(\Phi_{\cp})$, then $e(\Psi)\geq e(\Phi_{\cp}\Psi^2)$, hence $L\geq \nu$ implies that $2e(\Psi)\geq \nu$, which we have already seen results in a contradiction. Finally, suppose that $e(\Phi_{\cp})=e(\Psi)$. Then $e(\Psi)<\frac{\nu}{2}$ implies that $e(\Phi_{\cp})<\frac{\nu}{2}$ and consequently that $e(\Phi_{\cp}\Psi^2)<\frac{\nu}{2}$. But this means that $L = e(\Phi_{\cp}\Psi^2)+e(\Psi)<\nu$, contradicting the fact that $L\geq \nu$.

Case 3: $L=\mu$. This case cannot occur as we have assumed that $e(\Phi_{\cp}\Psi^2)=\mu$, meaning that $e(\Phi_{\cp}\Psi^2)+e(\Psi)>\mu$ by the non-triviality of $\Psi$.

We conclude that $e(\Phi_{\cp}\Psi^2)<\mu$ . Suppose first that $\mu>1$. Then Theorem \ref{theorem:threethree}(3) implies that $c(\p,\g)=0$. As in the proof of Proposition \ref{proposition:liesinnewformspace} we may easily show that $\f_\Psi=\g\mid A_{\p}$. But we've just shown that $\g\mid A_{\p}=\g$. Therefore $\f_\Psi$ is a newform and we're done.

We show that the case $\mu=1$ cannot occur. Indeed, suppose that $\mu=1$ (and hence $e(\Phi_{\cp}\Psi^2)=0$). Then $\g$ is a newform of $\mathscr{S}_k(\p{\mathcal{N}_0},\Phi)$. As $\f_\Psi\sim \g$, we also have $\f_{\Psi\overline{\Psi}}\sim \g_{\overline{\Psi}}$. Our hypotheses imply that $\nu\geq 3$, so Lemma \ref{lemma:vanishingcoeffs} implies that $\f=\f_{\Psi\overline{\Psi}}$; hence $\f\sim \g_{\overline{\Psi}}$. Theorem 6.1 of \cite{shemanske-walling} implies that $\g_{\overline{\Psi}}$ is a newform of $\mathscr{S}_k(\p^{2e(\Psi)}{\mathcal{N}_0},\Phi)$, hence Theorem 3.5 of \cite{shemanske-walling} implies that in fact we have $\f=\g_{\overline{\Psi}}$. By comparing the levels of $\f$ and $\g_{\overline{\Psi}}$, we see that this means that $2e(\Psi)=\nu$; i.e. $e(\Psi)=\frac{\nu}{2}$. We assumed that  $e(\Psi)<\frac{\nu}{2}$ however, so we obtain a contradiction, finishing our proof.\end{proof}

\begin{theorem}\label{theorem:innertwist} If $e(\Phi_{\cp})<\nu$ then 
$
\mathscr{S}_k^+(\cn,\Phi)=\mathscr{S}_k^+(\cn,\overline{\Phi}_{\cp}\Phi_{{\mathcal{N}_0}})^{\Phi_{\cp}}.
$

If $e(\Phi_{\cp})=\nu$ and $\f$ is a normalized newform in $\mathscr{S}_k^+(\cn,\overline{\Phi}_{\cp}\Phi_{\cn_0})$, then $$\textbf{f}_{\Phi_{\cp}}=\g-C(\p,\g)\cdot\g\mid B_{\p}$$ for some normalized newform $\g$ in $\mathscr{S}_k^+(\cn,\Phi)$. 
\end{theorem}

\begin{proof}
When $K=\mathbb Q$ this is Corollary 3.4 of \cite{HPS}.

Note first that the theorem is vacuously true when $e(\Phi_{\cp})=0$. We therefore assume that $e(\Phi_{\cp})\geq 1$. As a consequence, $\nu\geq 2$.

Let $\f\in \mathscr{S}_k^+(\cn,\overline{\Phi}_{\cp}\Phi_{{\mathcal{N}_0}})$ be a newform. Applying Proposition \ref{proposition:liesinnewformspace} shows that $\f_{\Phi_{\cp}}\in \mathscr{S}_k^+(\cn,\Phi)$ is a newform. As $\mathscr{S}_k^+(\cn,\overline{\Phi}_{\cp}\Phi_{\cn_0})$ is generated by newforms, we have the inclusion 

\begin{equation}\label{equation:thm1star}
\mathscr{S}_k^+(\cn,\overline{\Phi}_{\cp}\Phi_{{\mathcal{N}_0}})^{\Phi_{\cp}}\subset \mathscr{S}_k^+(\cn,\Phi).
\end{equation}

Now let $\f\in \mathscr{S}_k^+(\cn,\Phi)$. Then as above $\f_{\overline{\Phi}_{\cp}} \in \mathscr{S}_k^+(\cn,\overline{\Phi}_{\cp}\Phi_{{\mathcal{N}_0}})$ (by interchanging $\Phi_{\cp}$ and $\overline{\Phi}_{\cp}$ in equation \ref{equation:thm1star}), hence $\f_{\overline{\Phi}_{\cp}\Phi_{\cp}}\in \mathscr{S}_k^+(\cn,\overline{\Phi}_{\cp}\Phi_{{\mathcal{N}_0}})^{\Phi_{\cp}}$. This gives us the chain of inclusions

\begin{displaymath}
\mathscr{S}_k^+(\cn,\Phi)^{\overline{\Phi}_{\cp}\Phi_{\cp}}\subset \mathscr{S}_k^+(\cn,\overline{\Phi}_{\cp}\Phi_{{\mathcal{N}_0}})^{\Phi_{\cp}}\subset \mathscr{S}_k^+(\cn,\Phi).
\end{displaymath}

Lemma \ref{lemma:vanishingcoeffs} shows that $\mathscr{S}_k^+(\cn,\Phi)^{\overline{\Phi}_{\cp}\Phi_{\cp}}=\mathscr{S}_k^+(\cn,\Phi)$, and it follows that $$\mathscr{S}_k^+(\cn,\Phi)=\mathscr{S}_k^+(\cn,\overline{\Phi}_{\cp}\Phi_{{\mathcal{N}_0}})^{\Phi_{\cp}}.$$

We now prove the second assertion. Suppose that $e(\Phi_{\cp})=\nu$. First note that by Proposition \ref{proposition:levelbound}, $\f_{\Phi_{\cp}}\in \mathscr{S}_k(\p^{\nu+1}\cn_0,\Phi)$.
By Proposition \ref{proposition:eigentwist}, $\textbf{f}_{\Phi_{\cp}}$ is a Hecke eigenform for all $T_{\q}$ with $\q$ a prime not dividing $\cn$. Thus there exists an integer $\mu$ with $e(\Phi_{\cp})=\nu\leq \mu\leq \nu+1$ and a normalized newform $\g\in \mathscr{S}_k^+(\p^{\mu}\cn_0,\Phi)$ such that $\textbf{f}_{\Phi_{\cp}}\sim\g$. We claim that the case $\mu=\nu+1$ cannot occur. Indeed, if $\mu=\nu+1$ then $\g$ and $\textbf{f}_{\Phi_{\cp}}$ would both lie in $\mathscr{S}_k^+(\p^{\nu+1}\cn_0,\Phi)$ and our remarks at the end of Section \ref{section:prelims} would imply that $\textbf{f}_{\Phi_{\cp}}=\g$ is a newform. But Theorem \ref{theorem:threethree} shows that $C(\p,\f)\neq 0$, so that Corollary 6.4 of \cite{shemanske-walling} implies that  $\textbf{f}_{\Phi_{\cp}}$ is not a newform of any level. This contradiction allows us to conclude that $\mu=\nu$. It then follows from Proposition \ref{proposition:levelbound} that $\g_{\overline{\Phi}_{\cp}}\in\mathscr{S}_k(\p^{\nu+1}\cn_0,\overline{\Phi}_{\cp}\Phi_{\cn_0})$. Using the fact that $\g$ is an eigenform of $T_{\p}$ (as follows from Theorem 3.5 of \cite{shemanske-walling}), we see that  

$$\g-C(\p,\g)\cdot\g\mid B_{\p}=\g-\g\mid T_{\p}\mid B_{\p} = (\textbf g_{\overline{\Phi}_{\cp}})_{\Phi_{\cp}}=\left( c_1\f + c_2 \f\mid B_{\p} \right)_{\Phi_{\cp}}=c_1 \textbf f_{\Phi_{\cp}}$$

Comparing Fourier coefficients yields $c_1=1$.\end{proof}

\begin{theorem}\label{theorem:twocharacters}
If $0<e(\Psi)< \frac{\nu}{2}$ and $e(\Phi_{\cp})+e(\Psi)< \nu$ then $$\mathscr{S}_k^+(\cn,\Phi)^\Psi=\mathscr{S}_k^+(\cn,\Psi^2\Phi).$$
\end{theorem}
\begin{proof}
When $K=\mathbb Q$ this is Theorem 3.12 of \cite{HPS}.
We begin by noting that our hypotheses imply that $\nu\geq 3$. Let $\f\in \mathscr{S}_k^+(\cn,\Phi)$ be a newform. By Proposition \ref{proposition:liesinnewformspacearbitrary}, $\f_\Psi\in \mathscr{S}_k^+(\cn,\Psi^2\Phi)$ is a newform. As $\mathscr{S}_k^+(\cn,\Phi)$ is generated by newforms, we have the inclusion 

\begin{equation}\label{equation:thm2star}
\mathscr{S}_k^+(\cn,\Phi)^\Psi \subset \mathscr{S}_k^+(\cn,\Psi^2\Phi).
\end{equation}

Twisting by $\overline{\Psi}$ yields:

\begin{equation}\label{equation:thm2starstar}
\mathscr{S}_k^+(\cn,\Phi)^{\Psi\overline{\Psi}} \subset \mathscr{S}_k^+(\cn,\Psi^2\Phi)^{\overline{\Psi}}.
\end{equation}

We claim that $e(\Psi^2\Phi_{\cp})+e(\Psi)<\nu$. We have two cases to consider.

Case 1: $e(\Phi_{\cp})<\frac{\nu}{2}$ - By hypothesis $e(\Psi)< \frac{\nu}{2}$. Therefore $e(\Psi^2\Phi_{\cp})<\frac{\nu}{2}$, hence $e(\Psi^2\Phi_{\cp})+e(\Psi)<\nu$.

Case 2: $e(\Phi_{\cp})\geq \frac{\nu}{2}$ - We have two subcases to consider. Suppose first that $e(\Phi_{\cp})>e(\Psi^2)$. Then $e(\Psi^2\Phi_{\cp})=e(\Phi_{\cp})<\nu-e(\Psi)$. Now suppose that $e(\Phi_{\cp})\leq e(\Psi^2)$. Then $e(\Phi_{\cp})\leq e(\Psi^2)\leq e(\Psi)< \frac{\nu}{2}$. But Case 2 assumes that $e(\Phi_{\cp})\geq \frac{\nu}{2}$, so this subcase cannot occur and we have shown our claim.

Having shown that $e(\Psi^2\Phi_{\cp})+e(\Psi)<\nu$, we apply Theorem 5.7 of \cite{shemanske-walling} and Proposition \ref{proposition:liesinnewformspacearbitrary} to show that
\begin{equation}\label{equation:thm2starstarstar}
\mathscr{S}_k^+(\cn,\Psi^2\Phi)^{\overline{\Psi}} \subset \mathscr{S}_k^+(\cn,\Phi).
\end{equation}

Combining equations (\ref{equation:thm2starstar}) and (\ref{equation:thm2starstarstar}) gives us the chain of inclusions:

\begin{displaymath}
\mathscr{S}_k^+(\cn,\Phi)^{\Psi\overline{\Psi}} \subset \mathscr{S}_k^+(\cn,\Psi^2\Phi)^{\overline{\Psi}}  \subset \mathscr{S}_k^+(\cn,\Phi).
\end{displaymath}

Lemma \ref{lemma:vanishingcoeffs} implies that $ \mathscr{S}_k^+(\cn,\Phi)= \mathscr{S}_k^+(\cn,\Psi^2\Phi)^{\overline{\Psi}}$.

Twisting by $\Psi$ then yields:

$$\mathscr{S}_k^+(\cn,\Phi)^\Psi=\mathscr{S}_k^+(\cn,\Psi^2\Phi)^{\overline{\Psi}\Psi}.$$

As $e(\Psi^2\Phi_{\cp})<\nu$, Lemma \ref{lemma:vanishingcoeffs} shows that $\mathscr{S}_k^+(\cn,\Psi^2\Phi)^{\overline{\Psi}\Psi}=\mathscr{S}_k^+(\cn,\Psi^2\Phi)$, finishing the proof.\end{proof}

\begin{theorem} \label{theorem:primitivesum}
If $\frac{\nu}{2}<e(\Phi_{\cp})<\nu$ then $$\mathscr{S}_k^+(\cn,\Phi)=\bigoplus_{e(\Psi)=\nu-e(\Phi_{\cp})} \mathscr{S}_k^+(\mathfrak{p}^{e(\Phi_{\cp})}{\mathcal{N}_0},\Psi^2\Phi )^{\overline{\Psi}},$$ where the sum $\bigoplus_{e(\Psi)=\nu-e(\Phi_{\cp})} $ is taken over all Hecke characters $\Psi$ with conductor $\p^{\nu-e(\Phi_{\cp})}$ and infinite part $\Psi_{\infty}(a)=\mbox{sgn}(a)^l$ for $l\in\mathbb Z^n$ and $a\in K_{\infty}^\times$.
\end{theorem}
\begin{proof} When $K=\mathbb{Q}$ this is Theorem 3.9 of \cite{HPS}.

We begin by noting that our hypothesis $\frac{\nu}{2}<e(\Phi_{\cp})<\nu$ implies that $\nu\geq 2$. By Theorem \ref{theorem:threethree}(3) above and Theorem 6.8 of \cite{shemanske-walling} we have the inclusion $$\mathscr{S}_k^+(\cn,\Phi)\subset \sum_{e(\Psi)=\nu-e(\Phi_{\cp})} \mathscr{S}_k^+(\mathfrak{p}^{e(\Phi_{\cp})}{\mathcal{N}_0},\Psi^2\Phi)^{\overline{\Psi}}.$$ Our strategy to complete the proof will be to prove the reverse inclusion and then show that the sum is direct.

Let $\Psi$ be a Hecke character with conductor $\p^{\nu-e(\Phi_{\cp})}$ and infinite part $\Psi_{\infty}(a)=\mbox{sgn}(a)^l$,
 and let $\f\in \mathscr{S}_k^+(\mathfrak{p}^{e(\Phi_{\cp})}{\mathcal{N}_0},\Psi^2\Phi)$ be a newform. By Theorem 5.7 of \cite{shemanske-walling} we have $\f_{\overline{\Psi}}\in \mathscr{S}_k(\cn,\Phi)$ where $\cn$ is the exact level of $\f_{\overline{\Psi}}$. By Theorem \ref{theorem:threethree}(2), $C(\p,\f)\neq 0$, so by Theorem 6.3 of \cite{shemanske-walling}, $\f_{\overline{\Psi}}$ is a newform. Therefore for all $\p$-primary Hecke characters $\Psi$ with $e(\Psi)=\nu-e(\Phi_{\cp})$ we have the inclusion $$\mathscr{S}_k^+(\mathfrak{p}^{e(\Phi_{\cp})}{\mathcal{N}_0},\Psi^2\Phi)^{\overline{\Psi}}\subset \mathscr{S}_k^+(\cn,\Phi).$$ We have therefore shown that \begin{equation}\label{equation:sum}\mathscr{S}_k^+(\cn,\Phi)= \sum_{e(\Psi)=\nu -e(\Phi_{\cp})} \mathscr{S}_k^+(\mathfrak{p}^{e(\Phi_{\cp})}{\mathcal{N}_0},\Psi^2\Phi)^{\overline{\Psi}}.\end{equation}

It therefore remains only to show that the sum on the right hand side of equation \ref{equation:sum} is direct. We do this by showing that $$\mbox{dim}(\mathscr{S}_k^+(\cn,\Phi))=\sum_{e(\Psi)=\nu -e(\Phi_{\cp})} \mbox{ dim} (\mathscr{S}_k^+(\mathfrak{p}^{e(\Phi_{\cp})}{\mathcal{N}_0},\Psi^2\Phi)^{\overline{\Psi}}).$$

Given a Hecke character $\Psi$ with $e(\Psi)=\nu-e(\Phi_{\cp})$ and infinite part $\Psi_{\infty}(a)=\mbox{sgn}(a)^l$, fix a basis $S_{\Psi}$ of $\mathscr{S}_k^+(\mathfrak{p}^{e(\Phi_{\cp})}{\mathcal{N}_0},\Psi^2\Phi)$ consisting of normalized newforms $\f_1,\dots,\f_n$.

Define $$S=\bigcup_{\Psi} \{\textbf{f}_{\overline{\Psi}} : \f\in S_{\Psi} \}.$$

We have already shown that the elements of $S$ are all newforms of $\mathscr{S}_k^+(\cn,\Phi)$ and in fact span the space. It therefore suffices to show 
\begin{enumerate}
\item The (distinct) elements of $S$ are linearly independent
\item $\# S = \sum_{e(\Psi)=\nu-e(\Phi_{\cp})} \#S_\Psi =\sum_{e(\Psi)=\nu -e(\Phi_{\cp})} \mbox{ dim} (\mathscr{S}_k^+(\mathfrak{p}^{e(\Phi_{\cp})}{\mathcal{N}_0},\Psi^2\Phi)^{\overline{\Psi}}). $
\end{enumerate}

Note that (2) is equivalent to the statement that all the elements $\f_{\overline{\Psi}}$ of $S$ are distinct. 

We show that the elements of $S$ are linearly independent by assuming the contrary and obtaining a contradiction. Suppose that there is a nontrivial relation \begin{equation}\label{equation:relation}\sum_{i=1}^m c_i \textbf{h}_i=0\end{equation} where $\textbf{h}_i\in S$ (for all $i$), the $\textbf{h}_i$ are all distinct, and each $c_i$ is a non-zero scalar. Also assume that $m\geq 2$ is minimal in the sense that the elements of any subset of $S$ having fewer than $m$ elements are linearly independent.  

For a prime $\q$ which does not divide $\cn$, we can apply the linear operator $T_{\q}-C(\q,\textbf{h}_1) \mbox{Id}$ to equation \ref{equation:relation} to get

$$\sum_{i=1}^m c_i(C(\mathfrak{q},\textbf{h}_i)-C(\mathfrak{q},\textbf{h}_1)) \textbf{h}_i.$$

Note that the coefficient of $\textbf{h}_1$ is zero in the above sum. This means that the sum has fewer than $m$ summands and hence must be trivial by the minimality of $m$. As each $c_i$ is non-zero, we conclude that $C(\q,\textbf{h}_i)=C(\q,\textbf{h}_j)$ for all $1\leq i,j\leq m$ and $\q\nmid\cn$. As only finitely many primes divide $\cn$, Theorem 3.5 of \cite{shemanske-walling} shows that $\textbf{h}_1=\textbf{h}_2=\cdots=\textbf{h}_m$. This contradicts our assumption that the $\textbf{h}_i$ are distinct, proving that the elements of $S$ are linearly independent.

To prove that $$\# S = \sum_{e(\Psi)=\nu-e(\Phi_{\cp})} \#S_\Psi =\sum_{e(\Psi)=\nu -e(\Phi_{\cp})} \mbox{ dim} (\mathscr{S}_k^+(\mathfrak{p}^{e(\Phi_{\cp})}{\mathcal{N}_0},\Psi^2\Phi)^{\overline{\Psi}}),$$ it suffices to show if $\f\in \mathscr{S}_k^+(\mathfrak{p}^{e(\Phi_{\cp})}{\mathcal{N}_0},\Psi_0^2\Phi)$ and $\g\in \mathscr{S}_k^+(\mathfrak{p}^{e(\Phi_{\cp})}{\mathcal{N}_0},\Psi_1^2\Phi)$ are normalized newforms (with $\Psi_0,\Psi_1$ Hecke characters satisfying $e(\Psi_0)=e(\Psi_1)=\nu-e(\Phi_{\cp})$) such that $\f_{\overline{\Psi}_0}=\g_{\overline{\Psi}_1}$ then $\Psi_0=\Psi_1$ and $\f=\g$.

Suppose that $\f,\g$ are as in the previous paragraph and $\f_{\overline{\Psi}_0}=\g_{\overline{\Psi}_1}$. If $\Psi_0=\Psi_1$ then Theorem 3.5 of \cite{shemanske-walling} shows that $\f=\g$. Consequently, we may assume that $\Psi_0\neq \Psi_1$. Then $$\textbf{f}\mid A_{\p}=\textbf{f}_{\overline{\Psi}_0\Psi_0}=\textbf{g}_{\overline{\Psi}_1\Psi_0}.$$ Observe that $e(\Phi_{\cp}\Psi_1^2)=e(\Phi_{\cp})$ (as $e(\Phi_{\cp})>e(\Psi_1)$) and $0<e(\overline{\Psi}_1\Psi_0)\leq \mbox{ max} \{e(\Psi_1),e(\Psi_0)\}<\frac{\nu}{2}<e(\Phi_{\cp})$ by hypothesis. By Corollary 6.4 of \cite{shemanske-walling}, $\g_{\overline{\Psi}_1\Psi_0}\in \mathscr{S}_k^+(\p^{e(\Phi_{\cp})+e(\overline{\Psi}_1\Psi_0)}{\mathcal{N}_0},\Psi_0^2\Phi)$ is a normalized newform. As $\f\sim \f\mid A_{\p}$ and $\f\mid A_{\p}=\g_{\overline{\Psi}_1\Psi_0}$ we  must have $\f=\g_{\overline{\Psi}_1\Psi_0}$ (by Theorem 3.5 of \cite{shemanske-walling}). This means that $\f=\f\mid A_{\p}$. In particular, the $\p$-th coefficient of $\f$ is zero, contradicting Theorem \ref{theorem:threethree}(2) and finishing the proof.
\end{proof}

We conclude by presenting an application of the preceding theorems. This application makes clear the centrality of determining the vanishing of the $\p$-th `Fourier' coefficient of a Hilbert modular form in the study of character twists. This is a Hilbert modular analogue of Theorem 3.16 of \cite{HPS}.

Before stating the theorem however, we need a definition.

\begin{definition}
A newform $\g\in\mathscr{S}_k(\cn,\Phi)$ is said to be $\p$-primitive if $\g$ is not the twist of any newform of level $\cn^\prime$ where $\cn^\prime$ is a proper divisor of $\cn$ by a Hecke character by a Hecke character whose conductor is a power of $\p$.
\end{definition}

\begin{theorem}
Let $\f\in\mathscr{S}_k^+(\cn,\Phi)$ be a normalized newform. The following are equivalent:

\begin{enumerate}
\item $C(\p,\f)=0$
\item $\p^2\mid \cn$ and $e(\Phi_{\cp})<\nu$
\item $\f=\g_{\Psi}$ for some newform $\g$ in $\mathscr{S}_k^+(\cn^\prime,\Phi\overline{\Psi}^2)$ for some ideal $\cn^\prime$ dividing $\cn$ and some $\p$-primary Hecke character $\Psi$.
\end{enumerate}

Further, assuming (1), if $e(\Phi_{\cp})>\frac{\nu}{2}$ then in (3) $\g$ may be chosen so that $ord_{\p}(\cn^\prime)<ord_{\p}(\cn)$ and $\g$ is $\p$-primitive.
\end{theorem}

\begin{proof}
(1) implies (2) follows immediately from Theorem \ref{theorem:threethree}. Now assume (2) holds. We have two cases to consider. If $\Phi_{\cp}$ is trivial then let $\Psi$ be a $\p$-primary Hecke character with $0<e(\Psi)<\frac{\nu}{2}$. Theorem \ref{theorem:twocharacters} shows that $\mathscr{S}_k^+(\cn,\overline{\Psi}^2\Phi_{\cn_0})^{\Psi}=\mathscr{S}_k^+(\cn,\Phi_{\cn_0})$ and that there exists a newform $\g\in \mathscr{S}_k^+(\cn,\overline{\Psi}^2\Phi_{\cn_0})$ such that $\f=\g_{\Psi}$. Now suppose that $\Phi_{\cp}$ is nontrivial. Then Theorem \ref{theorem:innertwist} shows that there exists a newform $\g\in\mathscr{S}_k^+(\cn,\overline{\Phi}_{\cp}\Phi_{\cn_0})$ such that $\f=\g_{\Phi_{\cp}}$. We therefore take $\cn^\prime=\cn$ and $\Psi=\Phi_{\cp}$. Finally, assume (3) holds. Then $C(\p,\f)=C(\p,\g_{\Psi})=\Psi^*(\p)C(\p,\g)=0$ by Proposition \ref{proposition:crudebound}.

For the final assertion, note that $\frac{\nu}{2}<e(\Phi_{\cp})<\nu$ implies, by Theorem \ref{theorem:primitivesum}, that there exists a newform $\g\in\mathscr{S}_k^+(\p^{e(\Phi_{\cp})}\cn_0,\Psi^2\Phi)$ such that $\f=\g_{\overline{\Psi}}$, where $\Psi$ is a $\p$-primary Hecke character with $e(\Psi)=\nu-e(\Phi_{\cp})$. We show that such a $\g$ is $\p$-primitive. It clearly suffices to show that $C(\p,\g)\neq 0$, which follows from Theorem \ref{theorem:threethree} as $e(\Psi^2\Phi_{\cp})=e(\Phi_{\cp})=ord_{\p}(\p^{e(\Phi_{\cp})}\cn_0)$.\end{proof}

 %%%%%%%%%%%%%%%%%%%%
 %%%%%%%%%%%%%%%%%%%%
 %%%%%%%%%%%%%%%%%%%%
 %%%%%%%%%%%%%%%%%%%%
 %%%%%%%%%%%%%%%%%%%%
 %%%%%%%%%%%%%%%%%%%%
 %%%%%%%%%%%%%%%%%%%%

\end{document}